\documentclass[a4paper,12pt]{article}
\usepackage[utf8]{inputenc}
\usepackage{mathrsfs,amsfonts,amssymb,amsmath,array,times,mathptmx}

\newcommand\unit{{\mathbf{1}}}

\newcommand\sg[2]{\left(#2_{#1}\right)_{#1\ge0}}

\newtheorem{theorem}{Theorem}[section]
\newtheorem{proposition}[theorem]{Proposition}
\newtheorem{lemma}[theorem]{Lemma}
\newtheorem{definition}[theorem]{Definition}
\newtheorem{corollary}[theorem]{Corollary}
\newtheorem{remark}[theorem]{Remark}

\newenvironment{proof}[1][]{\begin{trivlist}\item[\hspace{3pt}
\textit{Proof#1.}]}{\rule{0.3em}{1.5ex}\end{trivlist}}
\newcommand{\set}[1]{\left\{ #1\right\}}
\newcommand{\norm}[1]{\left\| #1\right\|}
\newcommand{\scpro}[2]{\langle #1,#2\rangle}
\newcommand\uvtimes{{\mathbin{\mathop{~\otimes~}\limits_{u,v}}}}
\newcommand\uvptimes{{\mathbin{\mathop{~\otimes~}\limits_{u',v'}}}}
\newcommand\uvalltimes{{\mathbin{~\mathop{\otimes}\nolimits^0~}}}
\renewcommand\uvtimes{{{}_u{\otimes}_{v}}}
\renewcommand\uvptimes{{{}_{u'}{\otimes}_{v'}}}
\newcommand\E{\mathcal{E}}
\newcommand\F{\mathcal{F}}

\newcommand{\hl}[1]{\textbf{\textit{#1}}}
\newcommand{\sB}{\mathcal{B}}
\newcommand{\vt}{\vartheta}

\begin{document}

\bibliographystyle{amsalpha}

\title{\Large\bfseries {The Spatial Product of  Arveson Systems is Intrinsic}}
\author{B.V.\ Rajarama Bhat, Volkmar Liebscher, Mithun Mukherjee,\\and Michael Skeide
}
\date{June 2010
}

\maketitle
\begin{abstract}
\noindent
We prove  that the spatial product of two spatial Arveson systems is independent of the choice of the reference units. This also answers the same question for the minimal dilation the Powers sum of two spatial CP-semigroups: It is independent up to cocycle conjugacy.
\end{abstract}

\section{Introduction}

Arveson \cite{Arv89} associated with every \hl{$E_0$-semigroup} (a semigroup of unital endomorphisms) on $\sB(H)$ its \hl{Arveson system} (a family of Hilbert spaces $\E=\sg t\E$ with an associative identification $\E_s\otimes\E_t=\E_{s+t}$). He showed that $E_0$-semi\-groups are classified by their Arveson system up to cocycle conjugacy. By a \hl{spatial} Arveson system we understand a pair $(\E,u)$ of an Arveson system $\E$ and a \hl{unital unit} $u$ (that is a section $u=\sg t u$ of unit vectors $u_t\in\E_t$ that factor as $u_s\otimes u_t=u_{s+t}$). Spatial Arveson systems have an index, and this index is additive under the tensor product of Arveson systems.

Much of this can be carried through also for product systems of Hilbert modules and $E_0$-semigroups on $\sB^a(E)$, the algebra of all adjointable operators on a Hilbert module; see the conclusive paper Skeide \cite{Ske08p1} and its list of references. However, there is no such thing as the tensor product of product systems of Hilbert modules. To overcome this, Skeide \cite{Ske06d} (preprint 2001) introduced the product of spatial product systems (henceforth, the \hl{spatial product}), under which the index of spatial product systems of Hilbert modules is additive.

It is known that the spatial structure of a spatial Arveson system $\sg t\E$ depends on the choice of the \hl{reference unit} $\sg t u$. In fact, Tsirelson \cite{Tsi08} showed that if $\sg t v$ is another unital unit, then there need not exist an automorphism of $\sg t\E$ that sends $\sg t u$ to $\sg t v$. Also the spatial product depends \textit{a priori} on the choice of the reference units of its factors. This immediately raises the question if different choices of references units give isomorphic products or not. In these notes we answer this question in the affirmative sense for the spatial product of Arveson systems.

\vspace{2ex}
For two Arveson systems $\sg t\E$ and $\sg t\F$ with reference units ($\sg t u$ and $\sg t v$, respectively, their spatial product can be identified with the subsystem of the tensor product generated by the subsets $u_t\otimes\F_t$ and $\E_t\otimes v_t$. This raises another question, namely, if that subsystem is all of the tensor product or not. This has been answered in the negative sense by Powers \cite{Pow04}, resolving the same question for a related problem. Let us describe this problem very briefly.

Suppose we have two $E_0$-semigroups $\vt^i=\sg t{\vt^i}$ on $\sB(H^i)$ with \hl{intertwining} semigroups $\sg t{U^i}$ of isometries in $\sB(H^i)$ (that is, $\vt^i_t(a^i)U_t=U_ta^i_t$). Intertwining semigroups correspond one-to-one with unital units of the associated Arveson systems $\sg t{\E^i}$, so that these are spatial. Then by
$$
T
\begin{pmatrix}
a_{11}&a_{12}\\a_{21}&a_{22}
\end{pmatrix}
~:=~
\begin{pmatrix}
\vt^1_t(a_{11})&{U^1_t}^*a_{12}U^2_t\\{U^2_t}^*a_{21}U^1_t&\vt^2_t(a_{22})
\end{pmatrix}
$$
we define a Markov semigroup on $\sB(H^1\oplus H^2)$. Its unique \hl{minimal dilation} (see Bhat \cite{Bha96}) is an $E_0$-semigroup (fulfilling some properties). At the 2002 Workshop Advances in Quantum Dynamics in Mount Holyoke, Powers asked for the cocycle conjugacy class (that is, for the Arveson system) of that $E_0$-semigroup. More precisely, he asked if it is the cocycle conjugacy class of the tensor product of $\vt^1$ and $\vt^2$, or not. Still during the workshop Skeide (see the proceedings \cite{Ske03c}) identified the Arveson system of that \hl{Powers sum} as the spatial product system of the Arveson systems of $\vt^1$ and $\vt^2$. So, Powers' question is equivalent to the question if the spatial product is the tensor product, or not.

In \cite{Pow04} Powers answered the former question in the negative sense and, henceforth, also the latter. He left open the question if the cocycle conjugacy class of the minimal dilation of the Powers sum depends on the choice of the intertwining isometries. Our result of the present notes tells, no, it doesn't depend. We should say that Powers in \cite{Pow04} to some extent considered the Powers sum not only for $E_0$-semigroups but also for those CP-semigroups he called as \textit{spatial}. We think that his definition of spatial CP-semigroup is too restrictive, and prefer to use Arveson's definition \cite{Arv97}, which is much wider; see Bhat, Liebscher, and Skeide \cite{BLS10}. The definition of Powers sum easily extends to those CP-semigroups and the relation of the associated Arveson system of the minimal dilations is stills the same: The Arveson system of the sum is the spatial product of the Arveson systems of the addends; see Skeide \cite{Ske10}. Therefore, our result here also applies to the more general situation.

\begin{remark}
It should be noted that the result is visible almost at a glance when the intuition of random sets to describe spatial Arveson systems is available; see \cite{Lie09,Tsi00p1}. However, in order to make this clear a lot of random set techniques had to be explained, so we opted to give a plain Hilbert space proof. Although this is, maybe, not too visible, the proof here is nevertheless very much inspired by the intuition coming from random sets. We will explain that intuition elsewhere (\cite{BLS08p1}).
\end{remark}

\paragraph{Acknowledgments} This work began with an RiP stay at Mathematisches For\-schungsinstitut Oberwolfach in May 2007 and was almost completed during a stay of V.L.\ at Universit\`a del Molise, Campobasso, 2009. V.L.\  thanks M.S.\  for warm hospitality during the latter stay. M.S.\ is supported by funds from the Italian MIUR (PRIN 2007) and the Dipartimento S.E.G.e S.

 \section{Arveson  systems}

\begin{definition}  \label{defi:AS}
An \hl{Arveson  system} is a measurable family $\E=(\E_t)_{t\ge0}$ of separable Hilbert spaces  endowed with a measurable family of  
unitaries $V_{s,t}:\E_s\otimes \E_t\mapsto \E_{s+t}$ for all $s,t\ge0$ such that for all $r,s,t\ge0$
  \begin{displaymath}
    V_{r,s+t}\circ(\unit_{\E_r}\otimes V_{s,t})= V_{r+s,t}\circ( V_{r,s}\otimes\unit_{\E_t}).
  \end{displaymath}
  \end{definition}
 
 \begin{remark}\label{rem:notation}
In the sequel, we shall omit the $V_{s,t}$ and simply identify $\E_s\otimes\E_t$ with $\E_{s+t}$. This lightens the formulae, but requires a certain flexibility (see Proposition \ref{prop:inductive limit} or the proof of Lemma \ref{lem1}) when interpreting correctly operators on tensor products of Arveson systems.
\end{remark}
  
    \begin{remark}
Note that Definition \ref{defi:AS} is equivalent to Arveson's in \cite{Arv89}; see \cite[Lemma 7.39]{Lie09}. The only difference is that Definition \ref{defi:AS} allows for one-dimensional and zero-dimensional Arveson systems. The latter is necessary in view of the following property.
  \end{remark}

By \cite[Theorem 5.7]{Lie09}, for every Arveson system $\mathcal {E}$ the set
\begin{displaymath}
  \mathscr{S}(\mathcal{E}):=\set{\mathcal {F}: \mbox{product subsystem of }\mathcal{E}}
\end{displaymath}
forms a (complete) lattice with the lattice operations $\E'\wedge\F'=(\E'_t\cap\F'_t)_{t\ge0}$ and $\E'\vee\F'$ defined as the smallest Arveson subsystem containing both $\E'$ and $\F'$.

      \begin{remark}
By \cite[Theorem 7.7]{Lie09}, the algebraic structure of an Arveson system determines the measurable structure completely.   
  \end{remark}

  \begin{definition}
    A \hl{unit} $ u$ of an Arveson  system is a measurable non-zero section $\sg t u$ through $\sg t\E$, which satisfies for all $s,t\ge0$
    \begin{displaymath}
       u_{s+t}= u_s\otimes  u_t.
    \end{displaymath}
If $ u$ is \hl{unital} ($\norm{ u_t}=1\forall t\ge0$), the pair $(\E, u)$ is also called a \hl{spatial} Arveson system.
  \end{definition}

For Hilbert spaces, the spatial product from Skeide \cite{Ske06d} can be defined as a subsystem of the tensor product in the following way.

\begin{definition}
Let $(\E,u)$ and $(\F,v)$ be two spatial Arveson systems. We define their\hl{ spatial  product} as   
  \begin{displaymath}
    \E\uvtimes{}\F:=(u\otimes\F)\vee(\E\otimes v)\subset\E\otimes\F
  \end{displaymath}
\end{definition}

That this coincides with the product in \cite{Ske06d} follows either from the universal property \cite[Theorem 5.1]{Ske06d} that characterizes it, or after Proposition \ref{prop:inductive limit} below, that identifies directly the pieces from the inductive limit by which the product is constructed in \cite{Ske06d}.

Let
\begin{displaymath}
  \Pi^t:=\set{(t_1,\dots,t_n):n\in \set{1,2,\dots}, t_1>0, t_1+\dots+t_n=t}
\end{displaymath}
denote the set of interval partitions of $[0,t]$ (parametrized suitably for our purposes). For ${\bf t}=(t_1,\dots,t_n)\in\Pi^t$ and ${\bf s}=(s_1,\dots,s_m)\in\Pi^s$, denote by ${\bf t}\smallsmile{\bf s}:=(t_1,\dots,t_n,s_1,\dots,s_m)\in\Pi^{t+s}$ their \hl{join}.
We order $\Pi^t$ by saying that $(t_1,\dots,t_n)\prec(s_1,\dots,s_m)$ if there exist ${\bf t}_i\in\Pi^{t_i}$ such that \begin{displaymath}
{\bf t}_1\smallsmile\ldots\smallsmile{\bf t}_n={\bf s}.
\end{displaymath}

For any Hilbert subspace $ H$ denote by $H^\perp$ the orthogonal complement of $H$ in the space containing it. 
\begin{proposition}
\label{prop:inductive limit}
Let $(\E, u)$ and $(\F, v)$ be two spatial Arveson systems, and define
\begin{displaymath}
  G^{u,v}_t:= u_t\otimes v_t^\perp\oplus\mathbb{C} u_t\otimes v_t\oplus  u_t^\perp\otimes v_t.
\end{displaymath}
Then for all $t>0$
\begin{align*}\tag{$*$}\label{indlim}
    (\E\uvtimes{}\F)_t=\mathop{\mathrm{lim}}\limits_{(t_1,\dots,t_n)\in \Pi^t}G^{u,v}_{t_{n}}\otimes G^{u,v}_{t_{n-1}}\otimes\dots\otimes G^{u,v}_{t_2}\otimes G^{u,v}_{t_1}
  \end{align*}
\end{proposition}

\begin{proof}      
(See Remark \ref{rem:notation} about notation.) Since $G^{u,v}_t\subset\E_t\otimes\F_t$ and $G^{u,v}_{s+t}\subset G^{u,v}_s\otimes G^{u,v}_t$, the  limit exists due to monotonicity and $(\E\uvtimes\F)_t\subset\E_t\otimes\F_t\forall t\ge0$. From the properties of the interval partitions it is easy to see that in fact the RHS of \eqref{indlim} is a product system in its own right.  

Clearly,  $G^{u,v}_t\supset\E_t\otimes v_t$ and  $G^{u,v}_t\supset u_t\otimes F_t$. Therefore, the RHS of \eqref{indlim} contains both $\E\otimes v$ and $u\otimes \F$. 

On the other side, let $\mathcal{H}\subset \E\otimes \F$ contain both  $\E\otimes v$ and $u\otimes \F$. Then, obviously, $G^{u,v}_t\subset\mathcal{H}_t$. Consequently, $ \E\uvtimes{}\F$ contains the RHS of \eqref{indlim} and the assertion is proved.   
\end{proof}

\begin{remark}
The structure $G_s\otimes G_t\supset G_{s+t}$ is a recurrent theme in the analysis of quantum dynamics, in particular, of CP-semigroup; see \cite{MSchue93,BhSk00,BBLS04,Ske06d,MuSo02,	Mar03,Ske03c,BhMu10}. Recently, it has been formalized by Shalit and Solel \cite{ShaSo09} under the name of \hl{subproduct systems} (of Hilbert modules), and by Bhat and Mukherjee under the name of \hl{inclusion systems} (only the Hilbert case). Once for all, \cite{BhMu10} prove by the same inductive limit construction that every subproduct or inclusion system of Hilbert spaces embeds into an Arveson system. In Shalit and Skeide \cite{ShaSk10p}, the same will be shown for modules by reducing it to the case of CP-semigroups considered by Bhat and Skeide \cite{BhSk00}. While the spatial product may be viewed as \hl{amalgamation} of two spatial product systems over their reference units, \cite{BhMu10} generalize this to an \hl{amalgamation} over a contraction morphism between two (not necessarily spatial) Arveson systems. This applies, in particular, to the amalgamation of two spatial Arveson systems of not necessarily unital units, and answers Powers' question for the Markov semigroup obtained from non necessarily isometric intertwining semigroups.
\end{remark}

  \section{Universality of the spatial product}

Our aim is to prove the following theorem.

  \begin{theorem}\label{thm1}
Let $(\E, u)$, $(\E, u')$, $(\F, v)$ and  $(\F, v')$ be  spatial product systems. Then
\begin{displaymath}
  \E\uvtimes\F\cong\E\uvptimes\F.
\end{displaymath}
  \end{theorem}

Actually, we will prove even more, namely, $\E\uvtimes\F=\E\uvptimes\F$ as subsystems of $\E\otimes\F$. The key of the proof is the following lemma (whose proof we postpone to the very end, after having illustrated the immediate consequences).
  
\begin{lemma}\label{lem1}
\hfill
$\E\otimes v'\subset\E\uvtimes\F$.
\hfill{~~~~~~~~~~~~~~~~~~~~~}
\end{lemma}
  
\begin{corollary}
$\E{{}_{u}{\otimes}_{v'}}\F\subset\E\uvtimes\F$ and, by symmetry, $\E{{}_{u}{\otimes}_{v'}}\F\supset\E\uvtimes\F$, so $\E{{}_{u}{\otimes}_{v'}}\F=\E\uvtimes\F$. Once more, by symmetry $\E\uvptimes\F=\E{{}_{u}{\otimes}_{v'}}\F$.
\end{corollary}

This proves $\E\uvtimes\F=\E\uvptimes\F$ and, therefore, Theorem \ref{thm1}.
  
\begin{corollary}
    Denote by $\E^{0}$, $\F^{0}$ the product subsystems of $\E$ and $\F$ generated by all  units of $\E$ and $\F$ respectively. Then for the product with \hl{amalgamation over all units}
  \begin{displaymath}
    \E\uvalltimes\F:=\E\otimes\F^0\vee\E^0\otimes\F
  \end{displaymath}
we find  $\E\uvalltimes\F=\E\uvtimes\F$. 
  \end{corollary}
  
\begin{proof}
For every pair of unital units $u$ and $v$ we have
\begin{multline*}
\E\uvalltimes\F
=
\Bigl(\bigvee_{v'}\E\otimes v'\Bigr)\vee\Bigl(\bigvee_{u'}u'\otimes\F\Bigr)
\\
=
\bigvee_{v',u'}(\E\otimes v'\vee u'\otimes\F)
=
\bigvee_{v',u'}(\E\uvptimes\F)
=
\E\uvtimes\F,
\end{multline*}
because $\E\uvptimes\F=\E\uvtimes\F$.
\end{proof}

\begin{corollary}
 Suppose $\F$ is \hl{type I}, that is, $\F=\F^0$. Then $\E\uvtimes \F=\E\otimes \F$. 
  \end{corollary}
  
  \begin{proof}
$\E\uvtimes \F^0=\E\uvalltimes\F^0=\E\otimes\F^0\vee\E^0\otimes\F^0=\E\otimes \F^0$, because $\E\otimes\F^0\supset\E^0\otimes\F^0$.
\end{proof}

\begin{proof}[~of Lemma \ref{lem1}]
By Proposition \ref{prop:inductive limit}, it is enough to show that for $\psi\in \E_1$ we have
\begin{displaymath}
(\E\uvtimes \F)_1\ni  \lim_{n\to\infty}\mathrm{Pr}_{G_{2^{-n}}^{\otimes 2^n}}(\psi\otimes v'_1)=\psi\otimes v'_1
\end{displaymath}
which proves that $\E_1\otimes v'_1\subset(\E\uvtimes \F)_1$. (The proof of $\E_t\otimes v'_t\subset(\E\uvtimes \F)_t$ is analogous.)
Since $\mathrm{Pr}_{G_{2^{-n}}^{\otimes 2^n}}$ increases strongly to a projection (the projection onto $(\E\uvtimes \F)_1$), it is sufficient to show that
\begin{displaymath}
\lim_{n\to\infty}\norm{ \big.\smash{ \mathrm{Pr}_{G_{2^{-n}}^{\otimes 2^n}}(\psi\otimes v'_1})}=\norm{\psi}\norm{v'_1}=\norm{\psi}.
\end{displaymath}
For $0\le s< t\le 1$, we define the projections
\begin{displaymath}
  P_{s,t}:=\unit_{\E_s}\otimes \mathrm{Pr}_{u_{t-s}}\otimes \unit_{\E_{1-t}}\in \mathcal{B}(\E_1)
\end{displaymath}
in the factorization $\E_1=\E_s\otimes\E_ {t-s}\otimes \E_{1-t}$. We put $P_{t,t}:=\unit_{\E_1}$. Similarly, we define
\begin{displaymath}
  Q_{s,t}:=\unit_{\F_s}\otimes \mathrm{Pr}_{v_{t-s}}\otimes \unit_{\F_{1-t}}\in \mathcal{B}(\E_1).
\end{displaymath}
 Then
 \begin{multline*}
   \mathrm{Pr}_{(\E_s\otimes\F_s)\otimes G_{t-s}\otimes(\E_{1-t}\otimes\F_{1-t})}
\\
=P_{s,t}\otimes(\unit-Q_{s,t})+(\unit-P_{s,t})\otimes Q_{s,t}+P_{s,t}\otimes Q_{s,t}
\\
=(\unit-P_{s,t})\otimes Q_{s,t}+P_{s,t}\otimes \unit.
 \end{multline*}
(See Remark \ref{rem:notation} about notation!) This gives
\begin{align*}
\mathrm{Pr}_{G_{2^{-n}}^{\otimes 2^n}}
&
=
\prod_{i=1}^{2^n}\left((\unit-P_{\frac{i-1}{2^n},\frac{i}{2^n}})\otimes Q_{\frac{i-1}{2^n},\frac{i}{2^n}}+P_{\frac{i-1}{2^n},\frac{i}{2^n}}\otimes \unit\right)
\\
&
=
\sum_{S\subset\set{1,\dots,2^n}}
\Bigl(\prod_{i\in S}(\unit-P_{\frac{i-1}{2^n},\frac{i}{2^n}})\otimes Q_{\frac{i-1}{2^n},\frac{i}{2^n}}\Bigr)\Bigl(\prod_{i\notin S}P_{\frac{i-1}{2^n},\frac{i}{2^n}}\otimes \unit\Bigr)
\\\tag{$**$}\label{*}
&
=
\sum_{S\subset\set{1,\dots,2^n}}
\Bigl(\prod_{i\in S}(\unit-P_{\frac{i-1}{2^n},\frac{i}{2^n}})\prod_{i\notin S}P_{\frac{i-1}{2^n},\frac{i}{2^n}}\Bigr)\otimes \Bigl(\prod_{i\in S}Q_{\frac{i-1}{2^n},\frac{i}{2^n}}\Bigr).
\end{align*}
Since the $\scpro {v_t}{v'_t}$ form a (measurable) contractive semigroup, there is a complex number $\gamma$ with ${\sf Re}|\gamma|\le0$ such that $\scpro {v_t}{v'_t}=\mathrm{e}^{\gamma t}$. If we put
$$
w^S_i
:=
\begin{cases}
v_{\frac{1}{2^n}}&i\in S,
\\
v'_{\frac{1}{2^n}}&i\notin S,
\end{cases}
$$
then
$$
(\prod_{i\in S}Q_{\frac{i-1}{2^n},\frac{i}{2^n}}\Bigr)v'_1
=
\mathrm{e}^{\gamma \frac{\#S}{2^n}}(w^S_1\otimes\ldots\otimes w^S_{2^n}).
$$
Note that $w^S_1\otimes\ldots\otimes w^S_{2^n}$ are unit vectors. Note, too, that in the last line of \eqref{*} the projections $\prod_{i\in S}(\unit-P_{\frac{i-1}{2^n},\frac{i}{2^n}})\prod_{i\notin S}P_{\frac{i-1}{2^n},\frac{i}{2^n}}$ in the first factor are orthogonal for different choices of $S$. We conclude that
\begin{multline*}
\norm{ \big.\smash{ \mathrm{Pr}_{G_{2^{-n}}^{\otimes 2^n}}(\psi\otimes v'_1})}^2
=
\norm{ \Big.\smash{ \sum_{S\subset\set{1,\dots,2^n}}\mathrm{e}^{\gamma\frac{\#S}{2^n}}
\prod_{i\in S}(\unit-P_{\frac{i-1}{2^n},\frac{i}{2^n}})\prod_{i\notin S}P_{\frac{i-1}{2^n},\frac{i}{2^n}}\psi}}^2
\\[2ex]
=
\norm{ \Big.\smash{ \sum_{S\subset\set{1,\dots,2^n}}
\prod_{i\in S}\bigl(\mathrm{e}^{\gamma 2^{-n}}(\unit-P_{\frac{i-1}{2^n},\frac{i}{2^n}})\bigr)\prod_{i\notin S}P_{\frac{i-1}{2^n},\frac{i}{2^n}}\psi}}^2.
\end{multline*}
Next recall that $f(p)=f(0)\unit+(f(1)-f(0))p$ for every entire function $f$ and every projection $p$. We find
\begin{multline*}
\sum_{S\subset\set{1,\dots,2^n}}\prod_{i\in S}\bigl(\mathrm{e}^{\gamma 2^{-n}}(\unit-P_{\frac{i-1}{2^n},\frac{i}{2^n}})\bigr)\prod_{i\notin S}P_{\frac{i-1}{2^n},\frac{i}{2^n}}
\\
=
 \prod_{i=1}^{2^n}\left(\mathrm{e}^{\gamma 2^{-n}}(\unit-P_{\frac{i-1}{2^n},\frac{i}{2^n}})+P_{\frac{i-1}{2^n},\frac{i}{2^n}}\right)
\\
= 
\prod_{i=1}^{2^n}\left(\unit+(\mathrm{e}^{\gamma 2^{-n}}-1)(\unit-P_{\frac{i-1}{2^n},\frac{i}{2^n}})\right)
\\ 
=
\prod_{i=1}^{2^n}\mathrm{e}^{\gamma 2^{-n}(\unit-P_{\frac{i-1}{2^n},\frac{i}{2^n}})}
= 
\mathrm{e}^{\gamma 2^{-n}\sum_{i=1}^{2^n}(\unit-P_{\frac{i-1}{2^n},\frac{i}{2^n}})}.
\end{multline*}
From \cite[Proposition 3.18]{Lie09} (see also \cite[Proposition 8.9.9]{Arv03}), we know that $(s,t)\mapsto P_{s,t}$ is strongly continuous. The simplex $\set{(s,t):0\le s\le t\le1}$ is compact, so the function is even uniformly strongly continuous. This implies that $\unit-P_{s,t}\longmapsto 0$ strongly uniformly as $(t-s)\to 0$. Thus  we obtain that
\begin{displaymath}
  2^{-n}\sum_{i=1}^{2^n}(\unit-P_{\frac{i-1}{2^n},\frac{i}{2^n}})\xrightarrow{~n\to\infty~}\int_0^1(\unit-P_{t,t})\mathrm{d}t=0
\end{displaymath}
strongly.
Since entire functions are strongly continuous, this shows 
\begin{displaymath}
  \sum_{S\subset\set{1,\dots,2^n}}\prod_{i\in S}\bigl(\mathrm{e}^{\gamma 2^{-n}}(\unit-P_{\frac{i-1}{2^n},\frac{i}{2^n}})\bigr)\prod_{i\notin S}P_{\frac{i-1}{2^n},\frac{i}{2^n}}\xrightarrow{~n\to\infty~}\mathrm{e}^{0}=\unit
\end{displaymath}
in the strong topology, which completes the proof.  
\end{proof}


\newcommand{\Swap}[2]{#2#1}\newcommand{\Sort}[1]{}
\providecommand{\bysame}{\leavevmode\hbox to3em{\hrulefill}\thinspace}
\providecommand{\MR}{\relax\ifhmode\unskip\space\fi MR }
\providecommand{\MRhref}[2]{%
  \href{http://www.ams.org/mathscinet-getitem?mr=#1}{#2}
}
\providecommand{\href}[2]{#2}

\noindent
B.V.\ Rajarama Bhat: {\it Statistics and Mathematics Unit, Indian Statistical Institute Bangalore, R.\ V.\ College Post, Bangalore 560059, India},\\
E-mail: {\footnotesize\tt bhat@isibang.ac.in},\\
Homepage: {\footnotesize\tt http://www.isibang.ac.in/Smubang/BHAT/}

\vspace{2ex}\noindent
Volkmar Liebscher: {\it Institut f\"ur Mathematik und Informatik, Ernst-Moritz-Arndt-Universit\"at Greifswald, 17487 Greifswald, Germany},\\
E-mail: {\footnotesize\tt volkmar.liebscher@uni-greifswald.de},
Homepage:\\ {\scriptsize\tt  http://www.math-inf.uni-greifswald.de/mathe/index.php?id=97:volkmar-liebscher}

\vspace{2ex}\noindent
Mithun Mukherjee: {\it Statistics and Mathematics Unit, Indian Statistical Institute Bangalore, R.\ V.\ College Post, Bangalore 560059, India},\\
E-mail: {\footnotesize\tt mithun@isibang.ac.in}

\vspace{2ex}\noindent
Michael Skeide: {\it Dipartimento S.E.G.e S., Universit\`a degli Studi del Molise, Via de Sanctis, 86100 Campobasso, Italy},\\
E-mail: {\footnotesize\tt skeide@math.tu-cottbus.de},
Homepage:\\ {\footnotesize\tt http://www.math.tu-cottbus.de/INSTITUT/lswas/\_skeide.html}

\end{document}